\documentclass[11pt,reqno]{amsart}
\usepackage{indentfirst, amssymb,amsmath,amsthm}
\usepackage{xcolor}
\usepackage[colorlinks]{hyperref}
\newtheorem{theo}{Theorem}[section]

\newtheorem{lem}{Lemma}

\newtheorem{pro}{Proposition}[section]
\newtheorem{cor}{Corollary}[section]

\newtheorem{defi}{Definition}

\newtheorem{note}{Note}

\newcommand{\be}{\begin{equation}}
\newcommand{\ee}{\end{equation}}
\newcommand{\beas}{\begin{eqnarray*}}
\newcommand{\eeas}{\end{eqnarray*}}
\newcommand{\bea}{\begin{eqnarray}}
\newcommand{\eea}{\end{eqnarray}}

\numberwithin{equation}{section}

\begin{document}

\setlength{\unitlength}{1mm} \baselineskip .52cm
\setcounter{page}{1}
\pagenumbering{arabic}
\title[On Weak $G$-Completeness for Fuzzy Metric Spaces]{On Weak $G$-Completeness for Fuzzy Metric Spaces}

\author[Sugata Adhya and A. Deb Ray]{Sugata Adhya and A. Deb Ray}

\address{Department of Mathematics, The Bhawanipur Education Society College. 5, Lala Lajpat Rai Sarani, Kolkata 700020, West Bengal, India.}
\email {sugataadhya@yahoo.com}

\address{Department of Pure Mathematics, University of Calcutta. 35, Ballygunge Circular Road, Kolkata 700019, West Bengal, India.}
\email {debrayatasi@gmail.com}

\maketitle

\begin{abstract}
In this paper, we provide equivalent characterizations of weak $G$-complete fuzzy metric spaces. Since such spaces are complete, we also characterize fuzzy metric spaces that have weak $G$-complete fuzzy metric completions. Moreover we establish analogous results for classical metric spaces.
\end{abstract}


\section{\textbf{Introduction}}

Grabiec \cite{5} introduced $G$-Cauchy sequence as a weaker form of Cauchy sequence in the fuzzy context. He employed it to establish the celebrated Banach Contraction Principle for fuzzy metric spaces proposed by Kramosil and Michalek \cite{km}. $G$-Cauchyness was later adopted for fuzzy metrics in the context of Georege and Veeramani \cite{ver1}. The associated notion of completeness, known as $G$-completeness, has been extensively used to study fixed point theorems in fuzzy metric spaces. For details, one may consult \cite{f1, f2, 5, f3, f4}.

$G$-Cauchyness, being weaker than the usual Cauchyness, leads to a stronger completeness.  Unfortunately, $G$-completeness is even more stronger than it is desired to be, so that even a compact fuzzy metric space fails to be $G$-complete. To overcome this drawback Gregori, Mi$\tilde{\text{n}}$ana and Sapena introduced the notion of weak $G$-completeness \cite{wg}. They adopted and studied this new notion both for metric and fuzzy metric settings. In particular, they generalized Grabiec’s Banach Contraction Principle. Recently, in \cite{wgms} the authors characterized weak $G$-completeness by means of nested sequences of non-empty closed sets in the classical metric context.

It is worth noting, at this stage, that the class of weak $G$-complete (fuzzy) metric spaces lie between the classes of compact and complete (fuzzy) metric spaces. Metric spaces lying in this intermediate class have been an active research area in classical analysis over the years. Atsuji spaces \cite{at,pcs} and cofinally complete space \cite{pcs} are examples of such metric spaces. Further, the spaces lying in this intermediate class demand convergence of a class of sequences broader than the class of Cauchy sequences. Thus for Atsuji spaces we obtain the class of pseudo-Cauchy sequences \cite{pcs} whereas for cofinally complete spaces we obtain the class of cofinally Cauchy sequences \cite{3}.

The aim of this paper is to provide new characterizations for weak $G$-complete fuzzy metric spaces.  Here we characterize weak $G$-complete fuzzy metric spaces by means of the fuzzy metric analogue of pseudo-Cauchy and cofinally Cauchy sequences. Since a weak $G$-complete fuzzy metric space is complete, in what follows, we characterize those fuzzy metrics that have weak $G$-complete fuzzy metric completions. We also provide the classical metric analogue of our fuzzy metric characterizations for weak $G$-completeness.

\section{\textbf{Preliminaries}}

Throughout the paper the only notion of fuzzy metric we will be working on is the one due to George and Veeramani \cite{ver1, ver2} that goes as follows:

\begin{defi}
\normalfont A fuzzy metric space is an ordered triple $(X,M,*)$ where $X$ is a nonempty set, $*$ is a continuous $t$-norm and $M:X\times X\times(0,\infty)\to[0,1]$ is a mapping such that, for all $x,y,z\in X$ and $s,t>0,$ the following conditions hold:

a) $M(x,y,t)>0,$

b) $M(x,y,t)=1\iff x=y,$

c) $M(x,y,t)=M(y,x,t),$

d) $M(x,y,t)*M(y,z,s)\le M(x,z,t+s),$

e) $M(x,y,.):(0,\infty)\to[0,1]$ is continuous.

In this case, $(M,*)$ is said to be a fuzzy metric on $X.$
\end{defi}

\begin{lem}
\normalfont\cite{ver1} Given a fuzzy metric space $(X,M,*),$ $M(x,y,\cdot)$ defines a nondecresing map on $(0,\infty),$ $\forall~x,y\in X.$
\end{lem}

It has been shown in \cite{ver1} that every fuzzy metric $(M,*)$ on $X$ generates a first countable topology $\tau_M$ on $X$ such that $\{B(x,r,t):x\in X,r\in(0,1),t>0\}$ forms a base for $\tau_M,$ where $B(x,r,t)=\{y\in X:M(x,y,t)>1-r\},~\forall~x\in X,r\in(0,1),t>0.$ 

On the other hand, if $(X,d)$ is a metric space and $M_d:X\times X\times(0,\infty)\to[0,1]$ is defined by $M_d(x,y,t)=\frac{t}{t+d(x,y)},~\forall~x,y\in X,t>0,$ then $(X,M_d,\cdot)$ defines a fuzzy metric space $(\cdot$ being the usual multiplication of real numbers). Moreover, the topology $\tau(d)$ generated by the metric $d$ coincides with $\tau_{M_d}.$

\begin{theo}
\normalfont \cite{uni} Given a fuzzy metric space $(X,M,*),$ $(X,\tau_M)$ is metrizable.
\end{theo}

Let $(X,M, *)$ be a fuzzy metric space and $A\subset X.$ If $M_A=M|_{A\times A\times(0,\infty)},$ then $(A,M_A,*)$ defines a fuzzy metric space called the fuzzy metric subspace of $(X, M, *)$ on $A$ \cite{sub}. Clearly $\tau_{M_A}=(\tau_M)_A,$ $(\tau_M)_A$ being the subspace topology on $A$ induced by $\tau_M.$ 

$(X,M,*)$ is called precompact if for $r\in(0,1)$ and $t>0,$ there exists a finite subset $A$ of $X$ such that $X=\bigcup_{x\in A}{B_M(x,r,t)}$ \cite{uni}.

Convergence of sequences in $(X, M, *)$ is defined with respect to $\tau_M$. Thus a sequence $(x_n)$ in $(X,M,*)$ is said to be convergent to $x$ (\textit{resp.} clusters), if it does so in $(X,\tau_M)$ \cite{ver1}. 

\begin{theo}
\normalfont \cite{ver1} A sequence $(x_n)$ in a fuzzy metric space $(X,M,*)$ converges to $x\in X$ if and only if $\lim\limits_{n\to\infty}M(x_n,x,t)= 1,~\forall~t>0.$
\end{theo}

A sequence $(x_n)$ in a fuzzy metric space $(X,M,*)$ is called Cauchy if for $\epsilon\in(0,1),t>0,$ there exists $k\in\mathbb N$ such that  $M(x_m,x_n,t)>1-\epsilon,~\forall~m,n\ge k.$ It is easy to see that every convergent sequence in $(X,M,*)$ is Cauchy. As usual, $(X,M,*)$ is called complete if every Cauchy sequence in it converges \cite{ver1}.

The following proposition can be easily deduced.

\begin{pro}\label{pcau}
\normalfont Let $(X,d)$ be a metric space. Then 

a) A sequence $(x_n)$ is Cauchy in $(X,d)$ if and only if $(x_n)$ is Cauchy in $(X,M_d,\cdot).$

b) $A\subset X$ is complete as a metric subspace of $(X,d)$ if and only if $A$ is complete as a fuzzy metric subspace of $(X,M_d,\cdot).$
\end{pro}

Given two fuzzy metric spaces $(X,M,*)$ and $(Y,N,\star),$ a mapping $f:X\to Y$ is called an isometry if $M(x,y,t)=N(f(x),f(y),t),~\forall~x,y\in X,t>0.$ Moreover, if $f$ is onto then $(X,M,*)$ and $(Y,N,\star)$ are called isometric \cite{cpm}.

A fuzzy metric completion \cite{cpm} of $(X,M,*)$ is a complete fuzzy metric space such that $(X,M,*)$ is isometric to a dense subspace of it.

It is interesting to note that unlike metric spaces, a fuzzy metric space may not possess a fuzzy metric completion \cite{cpm}.

\begin{pro}
\normalfont \cite{cpm} Let $(X,d)$ be a metric space having completion $(\tilde{X},\tilde{d})$. Then, $(\tilde{X},M_{\tilde{d}},\cdot)$ is the unique (up to isometry) fuzzy metric completion of $(X,M_d,\cdot).$
\end{pro}

A sequence $(x_n)$ in a fuzzy metric space $(X,M,*)$ is called $G$-Cauchy if $\lim\limits_{n\to\infty}M(x_n,$ $x_{n+1},t)=0,~\forall~t>0$ \cite{5}. On the other hand, a sequence $(x_n)$ in a metric space $(X,d)$ is called $G$-Cauchy if $\lim\limits_{n\to\infty}d(x_n,x_{n+1})=0$ \cite{16}. A (fuzzy) metric space in which every $G$-Cauchy sequence converges is called a $G$-complete (fuzzy) metric space (\cite{5}, \cite{wg}). 

Unfortunately, this new notion of completeness is so strong that even compactness cannot imply $G$-completeness. To overcome this drawback, Gregori et. al. \cite{wg} introduced the following weaker version of completeness.

\begin{defi}
\normalfont A (fuzzy) metric space in which every $G$-Cauchy sequence clusters is called a weak $G$-complete (fuzzy) metric space.
\end{defi}

\begin{pro}\label{pr1}
\normalfont \cite{wg} Let $(X,d)$ be a metric space and $(x_n)$ a sequence in $X.$ Then

a) $(x_n)$ is $G$-Cauchy in $(X,d)$ if and only if $(x_n)$ is $G$-Cauchy in $(X,M_d,\cdot).$

b) $(X,d)$ is weak $G$-complete if and only if $(X,M_d,\cdot)$ is weak $G$-complete.
\end{pro}

A sequence $(x_n)$ in a fuzzy metric space $(X,M,*)$ is called fuzzy pseudo-Cauchy if for $\epsilon\in(0,1),t>0$ and $k\in\mathbb N$ there exist $p,q~(>k)\in\mathbb N$ with $p\ne q$ such that $M(x_p,x_q,t)>1-\epsilon$ \cite{sp}. On the other hand, a sequence $(x_n)$ in a metric space $(X,d)$ is called pseudo-Cauchy if for $\epsilon\in(0,1)$ and $k\in\mathbb N$ there exist $p,q~(>k)\in\mathbb N$ with $p\ne q$ such that $d(x_p,x_q)<\epsilon$ \cite{pcs}.

\begin{pro}\label{pr2}
\normalfont \cite{sp} Let $(X,d)$ be a metric space and $(x_n)$ a sequence in $X.$ Then $(x_n)$ is pseudo-Cauchy in $(X,d)$ if and only if $(x_n)$ is fuzzy pseudo-Cauchy in $(X,M_d,\cdot).$
\end{pro}

\section{\textbf{Main Results}}

We begin with the characterizations of weak $G$-complete (fuzzy) metric spaces. To meet our requirement, we first extend the notion of cofinally Cauchy sequences in fuzzy metric setting.

Howes \cite{cc1} introduced the notion of cofinally Cauchy sequence by replacing the condition of residuality with cofinality in the definition of
Cauchy sequence. A sequence $(x_n)$ in a metric space $(X,d)$ is called cofinally Cauchy if for $\epsilon>0$ there is an infinite subset $\mathbb N_\epsilon$ of $\mathbb N$ such that $d(x_p,x_q)<\epsilon,~\forall~p,q\in\mathbb N_\epsilon.$ If every cofinally Cauchy sequence in $(X,d)$ clusters, then $(X,d)$ is called cofinally complete.

\begin{defi}
\normalfont A sequence $(x_n)$ in a fuzzy metric space $(X,M,*)$ is said to be fuzzy cofinally Cauchy if for $\epsilon\in(0,1)$ and $t>0$ there is an infinite subset $\mathbb N_\epsilon$ of $\mathbb N$ such that $M(x_p,x_q,t)>1-\epsilon,~\forall~p,q\in\mathbb N_\epsilon.$ 
\end{defi}

The following is an easy consequence:

\begin{pro}\label{pr3}
\normalfont Let $(X,d)$ be a metric space and $(x_n)$ be a sequence in $X.$ Then $(x_n)$ is cofinally Cauchy in $(X,d)$ if and only if $(x_n)$ is cofinally Cauchy in $(X,M_d,\cdot).$
\end{pro}

\begin{theo}\label{th1}
\normalfont Let $(X,M,*)$ be a fuzzy metric space. Then the following conditions are equivalent:

(a) $(X,M,*)$ is weak $G$-complete.

(b) Each real-valued continuous function on $(X,\tau_M)$ carries a $G$-Cauchy sequence of $(X,M,*)$ to a cofinally Cauchy sequence of $\mathbb R$ (endowed with the usual metric).

(c) Each real-valued continuous function on $(X,\tau_M)$ carries a $G$-Cauchy sequence of $(X,M,*)$ to a pseudo-Cauchy sequence of $\mathbb R$ (endowed with the usual metric).
\end{theo}

\begin{proof}

\normalfont (a)$\implies$(b): Let $f:(X,\tau_M)\to\mathbb R$ be a continuous function. Choose a $G$-Cauchy sequence $(x_n)$ in $(X,M,*).$ Since $(X,M,*)$ is weak $G$-complete, $(x_n)$ clusters in $(X,\tau_M).$ Recall that $(X,\tau_M)$ is first countable. So there exists a subsequence $(x_{r_n})$ of $(x_n)$ that converges in $(X,\tau_M).$ Since $f$ is continuous, $(f(x_{r_n}))$ is Cauchy in $\mathbb R$, and consequently, $(f(x_n))$ is cofinally Cauchy in $\mathbb R.$

(b)$\implies$(c): Immediate.

(c)$\implies$(a): Let $(x_n)$ be a $G$-Cauchy sequence in $(X,M,*).$ If $(x_n)$ has a constant subsequence, then we are done. So, let us assume that $(x_n)$ has no constant subsequence. We first prove that $(x_n)$ has a $G$-Cauchy subsequence $(x_{r_n})$ in $(X,M,*)$ of distinct terms.

Set $r_1=1$ and $r_{n+1}=\max\{m\in\mathbb N:x_m=x_{r_n+1}\},~\forall~n\in\mathbb N.$ Since $(x_n)$ has no constant subsequence, $r_{n+1}$ exists, $\forall~n\in\mathbb N.$ Thus $(x_{r_n})$ defines a sequence of distinct terms.

Since $(x_n)$ is $G$-Cauchy, $\lim\limits_{n\to\infty}M(x_n,x_{n+1},t)=0,~\forall~t>0$ whence, $\lim\limits_{n\to\infty}M(x_{r_n},$ $x_{r_{(n+1)}},t)=0,~\forall~t>0.$ Thus $(x_{r_n})$ is a $G$-Cauchy subsequence of $(x_n)$ having distinct terms in $(X,M,*)$.

If possible, let $(x_{r_n})$ does not cluster in $(X,\tau_M)$. Then $A=\{x_{r_n}:n\in\mathbb N\}$ is a closed and discrete subset of $(X,\tau_M).$ Define $f:A\to\mathbb R$ by $f(x_{r_n})=2^n,~\forall~n\in\mathbb N.$ Clearly $f$ is continuous on $(X,\tau_M)$. Since $A$ is closed on $(X,\tau_M)$, by Tietze's extension theorem, $f$ extends to a continuous function $h$ on $(X,\tau_M).$ Note $(x_{r_n})$ is $G$-Cauchy in $(X,M,*)$ but $(h(x_{r_n}))$ is not pseudo-Cauchy in $\mathbb R$, a contradiction. Consequently $(x_{r_n}),$ and hence $(x_n),$ clusters in $(X,\tau_M).$

Thus $(X,d)$ is weak $G$-complete$.$
\end{proof}

In view of Proposition \ref{pr1}, the following corollary is obvious:

\begin{cor}
\normalfont Let $(X,d)$ be a metric space. Then the following conditions are equivalent:

(a) $(X,d)$ is weak $G$-complete.

(b) Each real-valued continuous function on $(X,d)$ carries a $G$-Cauchy sequence of $(X,d)$ to a cofinally Cauchy sequence of $\mathbb R$ (endowed with the usual metric).

(c) Each real-valued continuous function on $(X,d)$ carries a $G$-Cauchy sequence of $(X,d)$ to a pseudo-Cauchy sequence of $\mathbb R$ (endowed with the usual metric).
\end{cor}

\begin{theo}
\normalfont A closed subspace of a weak $G$-complete fuzzy metric space is weak $G$-complete.
\end{theo}

\begin{proof}
\normalfont Let $A$ be a closed subset of a weak $G$-complete fuzzy metric space $(X,M,*).$ Choose a $G$-Cauchy sequence $(x_n)$ in $(A,M_A,*).$ Then $(x_n)$ is $G$-Cauchy in $(X,M,*)$ and hence has a cluster point $c$ in $(X,\tau_M)$. Since $A$ is closed in $(X,\tau_M),$ so $c\in A.$ Thus $c$ becomes a cluster point of $(x_n)$ in $(A,M_A,*).$ Hence $(A,M_A,*)$ is weak $G$-complete.
\end{proof}

In view of Proposition \ref{pr1}, the following corollary is obvious:

\begin{cor}
\normalfont A closed subspace of a weak $G$-complete metric space is weak $G$-complete.
\end{cor}

Since a weak $G$-complete (fuzzy) metric space is complete, it is natural to ask under which conditions the completion of a (fuzzy) metric space is weak $G$-complete. In what follows, we give an answer to this. To establish the main result we require a lemma that involves the notion of Cauchy-continuous map for fuzzy metric spaces.

Recall that given two metric spaces $(X,d)$ and $(Y,\rho),$ a mapping $f:X\to Y$ is Cauchy-continuous if $f$ takes every Cauchy sequence of $X$ to a Cauchy sequence of $Y.$ The natural extension of this notion for fuzzy metric spaces is as follows.

\begin{defi}
\normalfont Let $(X,M,*)$ and $(Y,N,\star)$ be two fuzzy metric spaces and $A\subset X.$ A mapping $f:A\to Y$ is called fuzzy Cauchy-continuous if $f$ takes every Cauchy sequence of $(A,M_A,*)$ to a Cauchy sequence of $(Y,N,\star).$ 

Clearly if $f:A\to Y$ is fuzzy Cauchy-continuous, then $f$ is continuous as a mapping from $(A,\tau_{M_A})$ to $(Y,\tau_N).$ 
\end{defi}

\begin{lem}\label{prex}
\normalfont Let $A$ be a non-empty subset of a fuzzy metric space $(X,M,*)$ having fuzzy metric completion $(\tilde{X},\tilde{M},\tilde{*})$ and $f:(A,M_A,*)\to\mathbb R$ be a fuzzy Cauchy-continuous map. Then $f$ extends to a fuzzy Cauchy-continuous map $\overline{f}:(X,M,*)\to\mathbb R.$ (Here $\mathbb R$ is endowed with the standard fuzzy metric induced by the usual metric)
\end{lem}

\begin{proof}
\normalfont Let $\phi:(X,M,*)\to(\tilde{X},\tilde{M},\tilde{*})$ be an isometry such that $\phi(X)$ is dense in $(\tilde{X},\tau_{\tilde{M}}).$ Clearly $\phi$ is injective. 

Define $g:\phi(A)\to\mathbb R$ such that $g=f\phi^{-1}.$ Clearly $g$ is fuzzy Cauchy-continuous on $(\phi(A),\tilde{M}_{\phi(A)},\tilde{*}).$

We claim that $g$ extends to a fuzzy Cauchy-continuous map $g_*:\overline{\phi(A)}\to\mathbb R.$ 

Choose $b\in\overline{\phi(A)}.$ Since $(\tilde{X},\tau_{\tilde{M}})$ is first countable, there exists a sequence $(b_n)$ in $\phi(A)$ such that $\lim\limits_{n\to\infty}b_n=b$ in $(\tilde{X},\tau_{\tilde{M}}).$ Since $(b_n)$ is Cauchy in $\phi(A),$ so is $(g(b_n))$ in $\mathbb R,$ and consequently, $\lim\limits_{n\to\infty}g(b_n)$ exists.

Define $g_*:\overline{\phi(A)}\to\mathbb R$ by $g_*(c)=\lim\limits_{n\to\infty}g(c_n),~\forall~c\in\overline{\phi(A)}$ where $(c_n)$ is a sequence in $\phi(A)$ such that $\lim\limits_{n\to\infty}c_n=c$ in $(\tilde{X},\tau_{\tilde{M}}).$ Existence of such a sequence $(c_n)$ is ensured from the previous argument.

Note that $g_*$ is well-defined in the sense that for any two sequences $(r_n)$ and $(s_n)$ in $\phi(A)$ converging to the same point $d\in\overline{\phi(A)}$ we have $\lim\limits_{n\to\infty}g(r_n)=\lim\limits_{n\to\infty}g(s_n).$ Indeed $(r_1,s_1,r_2,s_2,r_3,s_3,\cdots)$ is Cauchy in $\phi(A)\implies\left(g(r_1),g(s_1),g(r_2),g(s_2),g(r_3),g(s_3),\cdots\right)$ is convergent in $\mathbb R,$ and consequently, $\lim\limits_{n\to\infty}g(r_n)=\lim\limits_{n\to\infty}g(s_n).$

We now show that $g_*$ is fuzzy Cauchy-continuous.

Let $(y^n)$ be a Cauchy sequence in $\overline{\phi(A)}.$ Then for each $n\in\mathbb N,$ there is a sequence $\left(x^n_k\right)_k$ in $\phi(A)$ such that $y^n=\lim\limits_{k\to\infty}x^n_k$ in $(\tilde{X},\tau_{\tilde{M}}).$ Consequently $g_*(y^n)=\lim\limits_{k\to\infty}g(x^n_k)$ in $\mathbb R,~\forall~n\in\mathbb N.$

So for each $n\in\mathbb N\backslash\{1\},$ there exists $p_n\in\mathbb N$ such that $\tilde{M}(y^n,x^n_k,\frac{1}{n})>1-\frac{1}{n},$ and $|g_*(y^n)-g(x^n_k)|<\frac{1}{n},$ $\forall~k\ge p_n.$

Set $z_n=x^n_{p_n},~\forall~n\in\mathbb N.$

Choose $\epsilon_0\in(0,1),t_0>0.$ Find $k\in\mathbb N$ such that $\frac{1}{k}<\min\{\epsilon_0,t_0\}.$

Then $\tilde{M}(y^n,z_n,t_0)\ge \tilde{M}(y^n,z_n,\frac{1}{n})>1-\frac{1}{n}>1-\epsilon_0,~\forall~n\ge k.$

Thus $\forall~t>0,$ $\tilde{M}(y^n,z_n,t)\to1$ and also $|g_*(y^n)-g(z_n)|\to0$ as $n\to\infty.$

Choose $\epsilon\in(0,1),t>0.$ Since $\tilde{*}$ is continuous, there exists $\delta\in(0,1)$ such that $(1-\delta)\tilde{*}(1-\delta)\tilde{*}(1-\delta)>1-\epsilon.$

Find $q\in\mathbb N$ such that $\tilde{M}(z_n,y^n,\frac{t}{3})>1-\delta$ and $\tilde{M}(y^m,y^n,\frac{t}{3})>1-\delta,~\forall~m,n\ge q.$

Then $\tilde{M}(z_m,z_n,t)\ge\tilde{M}(z_m,y^m,\frac{t}{3})\tilde{*}\tilde{M}(y^m,y^n,\frac{t}{3})\tilde{*}\tilde{M}(z_n,y^n,\frac{t}{3})\ge(1-\delta)\tilde{*}(1-\delta)\tilde{*}(1-\delta)>1-\epsilon,~\forall~m,n\ge q.$

Thus $(z_n)$ is Cauchy in $\phi(A)\implies(g(z_n))$ is Cauchy in $\mathbb R\implies (g_*(y^n))$ is Cauchy in $\mathbb R.$

Consequently $g_*$ is fuzzy Cauchy-continuous.

Since $g_*|_{\phi(A)}=g,$ so $g$ extends to a fuzzy Cauchy-continuous map $g_*:\overline{\phi(A)}\to\mathbb R.$ 

Then by Tietze extension theorem, $g_*$ extends to a continuous function $\overline{g}:\tilde{X}\to\mathbb R.$

Since $\tilde{X}$ is complete, $\overline{g}$ is Cauchy-continuous.

Let us now define $\overline{f}:X\to\mathbb R$ by $\overline{f}=\overline{g}\phi.$ Then $\overline{f}$ is clearly an extension of $f$ which is fuzzy Cauchy-continuous.
\end{proof}

\begin{theo}\label{th33}
\normalfont Let $(X,M,*)$ be a fuzzy metric space having a fuzzy metric completion $(\tilde{X},\tilde{M},\tilde{*})$. Then the following conditions are equivalent:

(a) $(\tilde{X},\tilde{M},\tilde{*})$ is weak $G$-complete.

(b) Every complete subset (as a fuzzy metric subspace) of $X$ is weak $G$-complete.

(c) Given any fuzzy metric space $(Y,N,\star)$ and a fuzzy Cauchy-continuous map $f:(X,M,*)\to(Y,N,\star),$ $f$ takes a $G$-Cauchy sequence of $(X,M,*)$ to a confinally Cauchy sequence of $(Y,N,\star)$.

(d) Given a pseudo Cauchy-continuous map $f:(X,M,*)\to\mathbb R$ where $\mathbb R$ is endowed with the standard fuzzy metric induced by the usual metric, $f$ takes a $G$-Cauchy sequence of $(X,M,*)$ to a confinally Cauchy sequence of $\mathbb R$.

(e) Every $G$-Cauchy sequence in $(X,M,*)$ has a Cauchy subsequence.
\end{theo}

\begin{proof}

\normalfont 

(a)$\implies$(b): Let $Y$ be a complete subset (as a fuzzy metric subspace) of $X$ and $\phi:(X,M,*)\to(\tilde{X},\tilde{M},\tilde{*})$ be an isometry such that $\phi(X)$ is dense in $(\tilde{X},\tau_{\tilde{M}}).$

Choose a $G$-Cauchy sequence $(y_n)$ in $(Y,M_Y,*).$ Then $(\phi(y_n)),$ being $G$-Cauchy in $(\tilde{X},\tilde{M},\tilde{*}),$ clusters to some point $c$ in $(\tilde{X},\tilde{M},\tilde{*}).$ So there is a subsequence $(\phi(y_{r_n}))$ of $(\phi(y_n))$ such that $\lim\limits_{n\to\infty}\phi(y_{r_n})=c$ in $(\tilde{X},\tilde{M},\tilde{*}),$ whence $\lim\limits_{n\to\infty}y_{r_n}=\phi^{-1}(c)$ in $(X,M,*).$ Since $Y$ is complete, so is $\phi(Y)$ (as a fuzzy metric subspace of $\tilde{X}$) whence $c\in\phi(Y).$ Thus $\phi^{-1}(c)\in Y.$ So $Y$ is weak $G$-complete.

(b)$\implies$(c): Let $(Y,N,\star)$ be a fuzzy metric space and $(x_n)$ be a $G$-Cauchy sequence in $(X,M,*).$ If possible, let $(x_n)$ has no Cauchy subsequence.

Then $A=\{x_n:n\in\mathbb N\}$ is complete as a fuzzy metric subspace and hence weak $G$-complete. Consequently, $(x_n)$ clusters in $X,$ a contradiction. Thus there exists a Cauchy subsequence $(x_{r_n})$ of $(x_n)$ in $(X,M,*).$

We first show that, $\{f(x_{r_n}):n\in\mathbb N\}$ is precompact as a fuzzy metric subspace of $(Y,N,\star)$. 

Suppose otherwise. Then there exists $\epsilon_0\in(0,1),t_0>0$ and a subsequence $(f(x_{m_{r_n}}))$ of $(f(x_{r_n}))$ such that $N(f(x_{m_{r_p}}),f(x_{m_{r_q}}),t_0)\le1-\epsilon_0,~\forall~p\ne q\cdots(*).$ However since $(x_{r_n})$ is Cauchy in $(X,M,*)$, so is $(f(x_{r_n}))$ in $(Y,N,\star),$ a contradiction to $(*)$ Hence $\{f(x_{r_n}):n\in\mathbb N\}$ is precompact. 

Choose $\epsilon\in(0,1),t>0.$ Since $*$ is continuous, there exists $\delta\in(0,1)$ such that $(1-\delta)*(1-\delta)>1-\epsilon.$

Since $\{f(x_{r_n}):n\in\mathbb N\}$ is precompact, there exists $y\in Y$ and an infinite subset $N_0$ of $\mathbb N$ such that $f(x_n)\in B_N(y,\delta,\frac{t}{2}),~\forall~n\in N_0.$ 

Thus $\forall~p,q\in N_0,$ $M(f(x_p),f(x_q),t)\ge M(f(x_p),y,\frac{t}{2})*M(f(x_q),y,\frac{t}{2})\ge(1-\delta)*(1-\delta)>1-\epsilon.$ 

So, $(f(x_n))$ is cofinally Cauchy.

(c)$\implies$(d): Immediate.

(d)$\implies$(e): Let $(x_n)$ be a $G$-Cauchy sequence in $(X,M,*).$ If $(x_n)$ has a constant subsequence, then we are done. So let us assume $(x_n)$ has no constant subsequence. Then proceeding as in Theorem \ref{th1}, we pass $(x_n)$ to a $G$-Cauchy subsequence having distinct terms.

If possible, let $(x_n)$ has no Cauchy subsequence in $(X,M,*).$ Let $A=\{x_n:n\in\mathbb N\}$ and $f:A\to\mathbb R$ be such that $f(x_n)=n,~\forall~x_n\in A.$ 

We first show that $f$ is fuzzy Cauchy-continuous as a mapping from $(A,M_A,*)$ to $\mathbb R.$ 

Let $(y_m)$ be a Cauchy sequence in $(A,M_A,*).$ If $(y_m)$ is eventually constant, then $(f(y_m))$ becomes eventually constant and hence Cauchy. So let us assume $(y_m)$ is not eventually constant. 

Choose $r_1=1.$ Since $(y_m)$ is Cauchy without being eventually constant, so for each $m\in\mathbb N$ there exists $r_{m+1}>r_m$ such that $y_{r_{(m+1)}}\ne y_{r_1},y_{r_2},\cdots,y_{r_m}.$ Thus $(y_{r_m})$ is a Cauchy subsequence of $(y_m)$ having distinct terms. Without loss of generality, let us pass $(y_m)$ to $(y_{r_m}).$

Note that $\exists~N_1\in\mathbb N$ such that $M(y_p,y_q,\frac{1}{2})>1-\frac{1}{2},~\forall~p,q\ge N_1$ and for chosen $N_r,$ $\exists~N_{r+1}~(>N_r)\in\mathbb N$ such that $M(y_p,y_q,\frac{1}{r+2})>1-\frac{1}{r+2},~\forall~p,q\ge N_{r+1}.$

Set $A_r=\{n\in\mathbb N:x_n=y_j\text{ for some }j\ge N_r\},~\forall~r\in\mathbb N.$ Then each $A_r$ is an infinite set of positive integers such that $A_r\supset A_{r+1},~\forall~r\in\mathbb N.$

Clearly $M(x_p,x_q,\frac{1}{r+1})>1-\frac{1}{r+1},~\forall~p,q\in A_r.$

For each $r\in\mathbb N,$ choose $n_r\in A_r$ such that $n_r<n_{r+1}.$ Then $(x_{n_r})$ is a Cauchy sequence in $(X,M,*)$. 

In fact for chosen $\epsilon\in(0,1),t>0$ there exists $r\in\mathbb N$ such that $\frac{1}{r+1}<\min\{\epsilon,t\}.$ Then $\forall~p,q\ge r,$ we have $n_p,n_q\in A_r,$ and consequently, $M(x_{n_p},x_{n_q},t)\ge M(x_{n_p},x_{n_q},\frac{1}{r+1})>1-\frac{1}{r+1}>1-\epsilon,~\forall~p,q\ge r.$ Thus $(x_{n_r})$ is Cauchy.

But it contradicts our assumption that $(x_n)$ has no Cauchy subsequence.

Hence every Cauchy sequence in $(A,M_A,*)$ must be eventually constant whence $f$ is fuzzy Cauchy-continuous.

Thus, in view of Lemma \ref{prex}, $f$ extends to a fuzzy Cauchy-continuous function from $(X,M,*)$ to $\mathbb R.$ So due to the hypothesis, $(f(x_n))$ must be cofinally Cauchy, a contradiction.

Hence the result follows.

(e)$\implies$(a): Let $(y_n)$ be a $G$-Cauchy sequence in $(\tilde{X},\tilde{M},\tilde{*})$ and $\phi:(X,M,*)\to(\tilde{X},\tilde{M},\tilde{*})$ be an isometry such that $\phi(X)$ is dense in $(\tilde{X},\tau_{\tilde{M}}).$ Then $\forall~n\in\mathbb N,~\exists~x_n\in X$ such that $\tilde{M}(\phi(x_n),y_n,\frac{1}{n+1})>1-\frac{1}{n+1}.$

Choose, $\epsilon\in(0,1),t>0.$ Since $\tilde{*}$ is continuous, there exists $\delta\in(0,1)$ such that $(1-\delta)\tilde{*}(1-\delta)\tilde{*}(1-\delta)>1-\epsilon.$

Since $(y_n)$ is $G$-continuous, there exists a positive integer $k >\max\left\{\frac{3}{t},\frac{1}{\delta}\right\}$ such that $\tilde{M}(y_n,y_{n+1},\frac{t}{3})>1-\delta,~\forall~n\ge k.$

Then $\forall~n\ge k,$ $M(x_n,x_{n+1},t)=\tilde{M}(\phi(x_n),\phi(x_{n+1}),t)\ge \tilde{M}(\phi(x_n),y_n,\frac{t}{3})\tilde{*}\tilde{M}(y_n,$ $y_{n+1},\frac{t}{3})\tilde{*}\tilde{M}(\phi(x_{n+1}),y_{n+1},\frac{t}{3})\ge \tilde{M}(\phi(x_n),y_n,\frac{1}{n+1})\tilde{*}\tilde{M}(y_n,y_{n+1},\frac{t}{3})\tilde{*}\tilde{M}(\phi(x_{n+1}),y_{n+1},$ $\frac{1}{n+2})
\ge(1-\delta)\tilde{*}(1-\delta)\tilde{*}(1-\delta)>1-\epsilon.$ 

Thus $(x_n)$ is $G$-Cauchy in $(X,M,*).$

Due to the hypothesis, $(x_n)$ has a Cauchy subsequence $(x_{r_n})$ in $(X,M,*),$ and hence $(\phi(x_{r_n}))$ is Cauchy in $(\tilde{X},\tilde{M},\tilde{*}).$ Let $\lim\limits_{n\to\infty}\phi(x_{r_n})=c$ in $(\tilde{X},\tilde{M},\tilde{*}).$

Then for any choice of $t>0,~\exists~p\in\mathbb N$ such that $\frac{t}{2}>\frac{1}{p+1},$ and hence $\forall~n\ge p,~\tilde{M}(\phi(x_{n}),y_{n},\frac{t}{2})\ge\tilde{M}(\phi(x_{n}),y_{n},\frac{1}{n+1})>1-\frac{1}{n+1}.$  

Since $\lim\limits_{n\to\infty}(1-\frac{1}{n+1})=1,$ it follows that $\lim\limits_{n\to\infty}\tilde{M}(\phi(x_{n}),y_{n},\frac{t}{2})=1,$ and hence $\lim\limits_{n\to\infty}\tilde{M}(\phi(x_{r_n}),y_{r_n},\frac{t}{2})=1.$ Thus $\lim\limits_{n\to\infty}\left[\tilde{M}(\phi(x_{r_n}),c,\frac{t}{2})\tilde{*}\tilde{M}(\phi(x_{r_n}),y_{r_n},\frac{t}{2})\right]=1.$ 

Since $\tilde{M}(y_{r_n},c,t)\ge \tilde{M}(\phi(x_{r_n}),c,\frac{t}{2})\tilde{*}\tilde{M}(\phi(x_{r_n}),y_{r_n},\frac{t}{2}),~\forall~n\in\mathbb N,$ it follows that $\lim\limits_{n\to\infty}\tilde{M}(y_{r_n},c,t)$ $=1.$ Thus $c$ is a cluster point of $(y_n)$ in $(\tilde{X},\tilde{M},\tilde{*}).$ 

Hence $(\tilde{X},\tilde{M},\tilde{*})$ is weak $G$-complete.
\end{proof}

In view of Propositions \ref{pcau}$-$\ref{pr1}, the following is obvious from Theorem \ref{th33}:

\begin{cor}\label{co33}
\normalfont \normalfont Let $(X,d)$ be a metric space. Then the followings conditions are equivalent:

(a) The completion of $(X,d)$ is weak $G$-complete.

(b) Every complete subset (as a metric subspace) of $X$ is weak $G$-complete.

(c) Given any metric space $(Y,\rho)$ and a Cauchy-continuous map $f:(X,d)\to(Y,\rho),$ $f$ takes a $G$-Cauchy sequence of $(X,d)$ to a confinally Cauchy sequence of $(Y,\rho)$.

(d) Given a Cauchy-continuous map $f:(X,d)\to\mathbb R$ where $\mathbb R$ is endowed with the usual metric, $f$ takes a $G$-Cauchy sequence of $(X,d)$ to a confinally Cauchy sequence of $\mathbb R$.

(e) Every $G$-Cauchy sequence in $(X,d)$ has a Cauchy subsequence.
\end{cor}

\begin{note}
\normalfont In theorem \ref{th33}, it is absolute necessary to assume the existence of fuzzy metric completion of $(X,M,*).$ For otherwise, we may obtain a fuzzy metric space that does not have a fuzzy metric completion, however every $G$-Cauchy sequence in it has a Cauchy subsequence. For instance, consider the following example:

Let $(x_n)_{n=3}^\infty$ and $(y_n)_{n=3}^\infty$ be two disjoint sequences of distinct points and $X=\{x_n:n\ge3\}\cup\{y_n:n\ge3\}.$ Define $M:X\times X\times(0,\infty)\to\mathbb R$ by $M(x_n,x_m,t)=M(y_n,y_m,t)=1-\left[\frac{1}{\min\{m,n\}}-\frac{1}{\max\{m,n\}}\right]$ and $M(x_n,y_m,t)=M(y_m,x_n,t)=\frac{1}{m}+\frac{1}{n},~\forall~m,n\ge3.$ If $*$ denotes the continuous $t$-norm defined by $a*b=\max\{0,a+b-1\},~\forall~a,b\in[0,1]$ then we know from \cite{cpm} that 

i) $(X,M,*)$ is a fuzzy metric space without having any fuzzy metric completion;

ii) $(x_n)_{n=3}^\infty$ and $(y_n)_{n=3}^\infty$ are Cauchy sequences in $(X,M,*).$

Since every subsequence of a Cauchy sequence is Cauchy, it is immediate to realize that every $G$-sequence in $X$ has a Cauchy subsequence, though $(X,M,*)$ has no fuzzy metric completion.
\end{note}

\begin{cor}
\normalfont Let $X$ be a (fuzzy) metric space having a (fuzzy) metric completion which is weak $G$-complete. Then every $G$-Cauchy sequence in $X$ is (fuzzy) cofinally Cauchy.
\end{cor}

\begin{proof}
\normalfont Immediate from the third conditions of Theorem \ref{th33} and Corollary \ref{co33} by considering $Y$ to be the space $X$ itself and $f$ to be the identity mapping on $X$.
\end{proof}

\begin{note}
\normalfont $(\sum_{i=1}^n\frac{1}{i})$ is a $G$-Cauchy sequence in $\mathbb R$ (endowed with the usual metric) which is not cofinally Cauchy. Hence $\mathbb R$ is not weak $G$-complete. Thus unlike cofinally complete metric  spaces \cite{3} a finite dimensional normed linear space may not be weak $G$-complete.
\end{note}

\end{document}